\documentclass[12pt,letterpaper]{article}
\usepackage[utf8]{inputenc}
\usepackage{epsfig,amsmath,amsthm,latexsym,amssymb,amsgen,graphicx}

\newtheorem{thm}{Theorem}[section]

\newtheorem{lem}[thm]{Lemma}
\newtheorem{cor}[thm]{Corollary}
\newtheorem{prop}[thm]{Proposition}
\theoremstyle{definition}

\newtheorem{prob}{Problem}

\newtheorem{rk}[thm]{Remark}

\theoremstyle{remark}

\newcommand{\Com}{\mathop{\mathrm{Com}}}
{\setcounter{tocdepth}{2}}
\newcommand{\lcm}{\mathop{\mathrm{lcm}}}

\begin{document}

\title{Matching in Power Graphs of Finite Groups}
\author{Peter J Cameron\footnote{School of Mathematics and Statistics, University of St Andrews, North Haugh, St Andrews, Fife, KY16 9SS, UK; pjc20@st-andrews.ac.uk, ORCiD 0000-0003-3130-9505},
Swathi V V\footnote{Department of Mathematics, National Institute of Technology Calicut, Kozhikode - 673601, India; swathivv14@gmail.com}\
and M S Sunitha\footnote{Department of Mathematics, National Institute of Technology Calicut, Kozhikode - 673601, India; sunitha@nitc.ac.in}}
\date{}
\maketitle

\begin{abstract}
The power graph $P(G)$ of a finite group $G$ is the undirected simple graph
with vertex set $G$, where two elements are adjacent if one is a power of
the other. In this paper, the matching numbers of power graphs of finite groups
are investigated. We give upper and lower bounds, and conditions for the power
graph of a group to possess a perfect matching. We give a formula for the
matching number for any finite nilpotent group. In addition, using some
elementary number theory, we show that the matching number of the enhanced
power graph $P_e(G)$ of $G$ (in which two elements are adjacent if both are
powers of a common element) is equal to that of the power graph of $G$.


\end{abstract}

\section{Introduction}
Associating graphs to algebraic structures is an interesting research topic. Cayley graphs, intersection graphs, zero divisor graphs, commuting graphs and power graphs  are some examples of graphs constructed from semigroups and groups.

 The directed power graph was first proposed in 2002 by Kelarev and Quinn \cite{ref6}. For a semigroup $S$ , the directed power graph $\overrightarrow P(S)$ is a graph with vertex set $S$ and there exists an arc from the vertex $x$ to the vertex $y$ if and only if $y=x^n$ for some natural number $n\in \mathbf{N}$. Motivated by this, Chakrabarty et al.~\cite{ref3} defined the undirected power graph of a semigroup. The undirected power graph $P(S)$ of a semigroup $S$ is a graph whose vertex set is $S$ and the edge set consists of pairs of distinct vertices $x$ and $y$ if $y=x^j$ or $x=y^k$ for some $j,k \in \mathbb{N}$.
 
 Many fascinating results on directed and undirected power graphs of finite semigroups and groups were established by several authors. The authors of \cite{ref3} proved that the power graphs of  ﬁnite groups are always connected. The authors also proved that the necessary and sufficient condition for $P(G)$ to be complete is that $G$ is a finite cyclic group having order $1$ or $p^i$, for some prime number $p$ and positive integer $i$.
 
 Curtin and Pourgholi \cite{ref4,ref5} proved that power graphs of cyclic groups have the largest clique and maximum number of edges, among all ﬁnite groups of a given order.
 
 In \cite{ref2}, the first author and Shamik Ghosh proved that finite abelian groups having isomorphic power graphs are isomophic. They also showed that Klein-4 group is the only finite group whose automorphism group is isomorphic to that of its power graph. In \cite{ref1}, the first author proved that finite groups having isomorphic power graphs have isomorphic directed power graphs.
 
   The chromatic number of power graphs of finite groups is investigated in \cite{ref7} and \cite{ref9} and some results on the independence number of the same is proved in \cite{ref8}. 
 
 A graph $\Gamma = (V,E)$ is an undirected simple graph $\Gamma$ with vertex set $V$ and edge set $E$.  A matching or independent edge set $M$ in a graph $\Gamma$ is a set of edges in which no two of them share a common vertex. A vertex is said to be matched (or saturated) if it is incident to one of the edges in the matching.  Otherwise the vertex is unmatched. A matching $M$ of a graph $\Gamma$ is a maximal matching if it is not a subset of any other matching in $\Gamma$. A matching that contains the largest possible number of edges is called a maximum matching. The size of a maximum matching in a graph $\Gamma$ is known as its matching number, which is denoted by $\mu(\Gamma)$. A perfect matching is a matching which saturates all vertices of the graph. 
 
 Let $G$ be a finite group. Together with the power graph, the enhanced power graph and the commuting graph are some of the examples of graphs whose vertex set is $G$ and whose edges reflect the group structure in some way. 
In the \emph{enhanced power graph} of $G$,denoted by $P_e(G)$, two vertices $x$ and $y$ are adjacent if and only if $\langle x,y \rangle$ is cyclic, and in the \emph{commuting graph} $\Com(G)$ of $G$, two vertices are adjacent if they commute. 
 
We denote the order of a group $G$ by $|G|$, while for $a\in G$,
the order of the element $a$ is denoted by $o(a)$.

 In this paper, we concentrate on finding matchings in the power graphs. We investigate several class of groups to obtain groups whose power graphs have a perfect matching.
In particular, we find the size of a maximum matching in the
power graph of any abelian group.
We also include some results we obtained on the enhanced power graph as well as commuting graph. In particular, and a little surprisingly, the power graph and
enhanced power graph have the same matching number.

\section{A preliminary result}

We begin this section by noting that a finite group $G$ of odd order has 
matching number $(|G|-1)/2$; that is, a maximum matching leaves just one
vertex unmatched. To see this, note that, for any $x\in G$, if $x\ne1$, then
$x\ne x^{-1}$ and $\{x,x^{-1}\}$ is an edge of $P(G)$; these edges form a
matching of the required size.

For groups of even order, we begin with the following observation.

\begin{thm}
Let $G$ be a group of even order. Let $T=\{g\in G:g^2=1\}$ be the set
consisting of the identity and the involutions in $G$. Let $X(G)$ be a graph
with vertex set $G$ with the property that every element of $G\setminus T$ is
joined to its inverse. Then there is a maximum-size matching in $G$ for which
the set of unmatched vertices is contained in $T$.
\end{thm}

\begin{proof}
Clearly $|T|>1$.

Take any matching $M$ of $X(G)$. We describe a transformation to another
matching $M'$ such that either $|M'|>|M|$, or $|M'|=|M|$ and the number of
unmatched vertices not in $T$ is smaller in $M'$ than in $M$.

Suppose that $g$ is an unmatched vertex which is not in $T$. If $g^{-1}$ is
also unmatched then we can match $g$ to $g^{-1}$, increasing the size of the
matching. So suppose $g^{-1}$ is matched. 

Put $g=g_0$. Let $g_1$ be the vertex matched to $g_0^{-1}$; let $g_2$ be the
vertex matched to $g_1^{-1}$; and so on, as long as possible. The process
terminates when either $g_m^{-1}$ is unmatched, or $g_m\in T$.

In the first case, we replace the edges $\{g_0^{-1},g_1\}$, $\{g_1^{-1},g_2\}$,
\dots, $\{g_{m-1}^{-1},g_m\}$ with the edges $\{g_0,g_0^{-1}\}$, 
$\{g_1,g_1^{-1}\}$, \dots,$\{g_m,g_m^{-1}\}$, and the size of the matching
is increased by one.

In the second case, we replace the edges $\{g_0^{-1},g_1\}$, $\{g_1^{-1},g_2\}$,
\dots, $\{g_{m-1}^{-1},g_m\}$ with the edges $\{g_0,g_0^{-1}\}$,
$\{g_1,g_1^{-1}\}$, \dots, $\{g_{m-1},g_{m-1}^{-1}\}$. The resulting matching
has the same size, but we have replaced the unmatched vertex $g_0\notin S$
by $g_m\in S$, so we have decreased by one the number of unmatched vertices
not in $T$.

Continuing this process, we find a matching of maximal size in which all
unmatched vertices belong to $T$.
\end{proof}

If a matching has the property of the theorem, and the identity is unmatched,
then we may add the edge $\{1,t\}$ to the matching, where $t$ is an involution.
So the following holds.

\begin{cor}
\begin{enumerate}
\item
Let $G$ be a group of even order. Then the size of a maximum matching in
the power graph or enhanced power graph of $G$ is at least $1+(|G|-|T|)/2$,
where $T$ consists of the identity and the involutions in $G$.
\item 
If $G$ is a group with a unique involution, then $P(G)$ has a perfect matching.
\end{enumerate}
\end{cor}

\begin{rk}
Groups with a unique involution are known; see~\cite{bc}.
\end{rk}

\section{Upper and lower bounds}

In this section, we describe upper and lower bounds for the matching number
of the power graph of a group of even order. First, an upper bound.

\begin{thm}
Let $G$ be a finite group of even order. Let $I(G)$ be the set of involutions
in $G$, and $O(G)$ the set of elements of odd order. Then
\begin{enumerate}
\item any matching of $P(G)$ leaves at least $|I(G)|-|O(G)|$ vertices unmatched;
\item if $G$ has a perfect matching, then $|I(G)|\le|O(G)|$.
\end{enumerate}
\label{t:mp1}
\end{thm}

\begin{proof}
Let $\Gamma$ be the induced subgraph of $P(G)$ on $G\setminus O(G)$ (the
set of elements of even order in $G$). For $t\in I(G)$, let
\[C_t=\{x\in G:t\in\langle x\rangle\}.\]
Note that elements of $C_t$ have even order, and no element of $G$ can lie in
more than one of these sets, since a cyclic group contains at most one
involution.

We will show that the sets $C_t$ for $t\in I(G)$ are connected components
of $\Gamma$, and that they all have odd cardinality. It follows from Tutte's
$1$-factor theorem~\cite{tutte} that, if $P(G)$ has a perfect matching, then
$|I(G)|$ (the number of odd components of the induced subgraph on
$G\setminus O(G)$) does not exceed $|O(G)|$. Moreover, the deficit form of
the theorem shows that, if $|I(G)|>|O(G)|$, there are at least $|I(G)|-|O(G)|$
vertices uncovered in any matching.

Note that any element of $C_t$ is joined to $t$ in the
power graph, so any two elements of $C_t$ have distance at most~$2$; thus
$C_t$ is contained in a connected component. Take an edge $\{x,y\}$ of the
power graph contained in $G\setminus O(G)$. Without loss of generality, $x$
is a power of $y$. Suppose that $t$ is the involution in $\langle x\rangle$,
so that $x\in C_t$. Then $t\in\langle x\rangle\le\langle y\rangle$, so also
$y\in C_t$. This shows that $C_t$ is a connected component of $\Gamma$.

Now all elements of $C_t\setminus\{t\}$ have order greater than $2$; so they
can be paired with their inverses, leaving only $t$ unpaired. So $|C_t|$ is
odd, as required.
\end{proof}

Now we give a lower bound.

\begin{thm}
Let $G$ be a finite group of even order. Let $S=I(G)$ be the set of involutions
in $G$, and $O(C_G(S))$ the set of elements of odd order which commute with all
involutions.
\begin{enumerate}
\item There is a matching leaving at most $\max\{0,|I(G)|-|O(C_G(S))|\}$
vertices unmatched.
\item If $|I(G)|\le|O(C_G(S))|$, then $P(G)$ has a perfect matching.
\end{enumerate}
\label{t:mp2}
\end{thm}

\begin{proof}
Let $n=|I(G)|$ and $m=|O(C_G(S))$. Suppose first that $m\ge n$.
Suppose first that $m\ge n$.
We start as usual with the matching $M$ on $G$ in which each element of order
greater than $2$ is matched to its inverse, leaving the identity and the
involutions unmatched. In addition, we match the identity to one of the
involutions. This leaves $n-1$ unmatched involutions, and $m-1$ elements of
odd order commuting with them, falling into $(m-1)/2$ inverse pairs. So we
can partition the unmatched involutions into $(n-1)/2$ pairs, and choose an
inverse pair of elements of odd order commuting with each pair of involutions.

Let $u,v$ be involutions, and $x,x^{-1}$ the corresponding pair of elements
of odd order commuting with $u$ and $v$. In the given matching, we have
edges $\{x,x^{-1}\}$, $\{ux,ux^{-1}\}$, and $\{vx,vx^{-1}\}$. We delete these
and include instead the edges $\{u,ux\}$, $\{v,vx^{-1}\}$,
$\{ux^{-1},x^{-1}\}$ and $\{vx,x\}$. Now all previously matched elements are
still matched, and in addition $u$ and $v$ are matched.

Repeating for all pairs of involutions we obtain a perfect matching.

\medskip

Now suppose that $m<n$. Proceeding as above, we can match $(m-1)/2$ pairs
of involutions with elements of odd order, leaving $n-m$ involutions
unmatched, as required.
\end{proof}

With these results we can calculate the matching number of the power graph of
a nilpotent group.

\begin{thm}
Let $G$ be nilpotent; let $I(G)$ and $O(G)$ be the sets of involutions and
elements of odd order respectively.
\begin{enumerate}
\item If $|I(G)|<|O(G)|$, then $G$ has a perfect matching.
\item Otherwise, a maximum matching leaves $|I(G)|-|O(G)|$ vertices unmatched.
\label{t:nilp}
\end{enumerate}
\end{thm}

\begin{proof}
If $G$ is nilpotent, then the elements of odd order form a normal subgroup
$O(G)$, and $G\cong H\times O(G)$ where $H$ is a Sylow $2$-subgroup. So all
involutions commute with all elements of odd order. So the result follows
from Theorems~\ref{t:mp1} and~\ref{t:mp2}.
\end{proof}

\section{Related results}

In this section we give some miscellanous related results.

\subsection{Groups whose power graph has small matching number}

\begin{thm}
For every positive integer $m$, there are only finitely many finite groups
$G$ with $\mu(G)=m$, apart from elementary abelian $2$-groups
(with $\mu(G)=1$); such a group satisfies $|G|<8m+4$.
\end{thm}

\begin{proof}
If $|G|$ is odd, then $m=\mu(G)=(|G|-1)/2$, so $|G|=2m+1$. So suppose that
$|G|$ is even. Then $|O(C_G(S))|\ge1$, and the number of vertices uncovered
in a maximum matching is $|G|-2m$. So
\[|I(G)|-1\ge|G|-2m,\]
whence $|I(G)\ge|G|-2m+1$. However, if $G$ is not elementary abelian, then
$|I(G)|<\frac{3}{4}|G|$ (This result is described in the literature as an
``easy exercise''). So $|G|<8m+4$.
\end{proof}

Using this, we can give the determination of groups whose power graph has
matching number $1$ or $2$. The \emph{dihedral group} $D_n$ is the group of
order $2n$ which is the symmetry group of a regular $n$-gon, for $n\ge3$.

\begin{thm}
Let $G$ be a finite group.
\begin{enumerate}
\item If $\mu(P(G))=1$, then $G$ is either an elementary abelian
$2$-group or $C_3$.
\item If $\mu(P(G))=2$, then $G$ is one of the following groups: $C_4$, $C_5$,
$D_3$ or $D_4$.
\end{enumerate}
\end{thm}

\begin{proof}
Part (a) follows immediately from the preceding theorem. For part (b), we know
that such a group has order at most~$11$, and there are only a small number of
groups to analyse.
\end{proof}

\subsection{Groups with few involutions}

We have seen that, if $G$ has a unique involution, then $P(G)$ has a perfect
matching. We now extend this result.

\begin{thm}
Let $G$ be a group with exactly three involutions, not all pairs of which
commute. Then either $G\cong S_3$, or $P(G)$ has a perfect matching.
\end{thm}

\begin{proof}
 Let $s,t,u$ be the involutions. If $s$ and $t$ do not
commute, then $\langle s,t\rangle$ is a dihedral group of order $2n$ containing
$n$ involutions, with $n\ge 3$; so we must have $n=3$, and $s,t,u$ are the
involutions in a normal subgroup of $G$ isomorphic to $S_3$.
Now $S_3$ is a complete group: this means that its centre and
its outer automorphism group are both trivial. Hence every extension of $S_3$
splits: that is, if $S_3$ is a normal subgroup of $G$, then
$G\cong S_3\times H$. See \cite[Section 13.5]{robinson}. Now $H$ contains no
involutions, so has odd order. If $|H|=1$, then $G\cong S_3$; otherwise
$H=O(C_G(S))$, and the result follows from Theorem~\ref{t:mp2}.
\end{proof}

For a group $G$, since $E(P(G))\subseteq E(P_e(G))\subseteq E(\Com(G))$, the possibility to have a perfect matching in the commuting graph is greater as compared to the power graph. The following theorem shows that, if the order of a group is much bigger than the number of involutions in it, then its commuting graph has a perfect matching.

\begin{prop}
There is a function $F$ such that, if $G$ is a group of even order which
has exactly $n$ involutions, and $|G|\ge F(n)$, then the commuting graph
of $G$ has a perfect matching.
\end{prop}

\begin{proof} We take $F(n)=2n\cdot n!$. So let $G$ be a group with even
order greater than $2n\cdot n!$ and suppose that $G$ contains $n$ involutions.
We begin with a matching $M$ as follows: elements of order greater than $2$
are matched to their inverses; the identity is mapped to one involution. If
$n=1$ we are finished, so suppose not.

The group $G$ acts by conjugation on the set $S$ of involutions. The kernel
of this action, which is $C_G(S)$, has index at most $n!$ in $G$, and so has
order at least $2n$; so , putting $X=C_G(S)\setminus(\{1\}\cup S)$, we have
$|X|\ge n-1$. Moreover, elements of $X$ have order greater than~$2$, and so are
matched with their inverses in $M$; and $X$ is inverse-closed. 

Pick $(n-1)/2$ inverse pairs in $X$, say $\{x_1,x_2\}$, \dots,
$\{x_{n-2},x_{n-1}\}$. Let $t_1,\ldots,t_{n-1}$ be the unmatched involutions
in $S$. Now delete the edges $\{x_{2i-1},x_{2i}\}$ from $M$ for
$i=1,\ldots,(n-1)/2$, and add the edges $\{x_1,t_2\}$, $\{x_2,t_2\}$, \dots,
$\{x_{n-1},t_{n-1}\}$ instead. (These are edges since $t_i\in S$ and
$x_i\in C_G(S)$.) The result is a perfect matching $M'$.
\end{proof}

The hypothesis in the above theorem is not enough in the case of power graphs, since, the power graph of $C_{2^n}\times C_{2^m}$ has no perfect matching even if we take $n$ and $m$ very large: the group has three involutions and one element
of odd order.

\begin{prop}
There is a function $F$ on the natural numbers with the following property:
Let $G$ be a finite group of even order, and $S$ the set of involutions in $G$.
Suppose that for every involution $u\in S$, there is an involution $v$ in $S$
which does not commute with $u$. If $|G|\ge F(|S|)$, then the power graph of
$G$ has a perfect matching.
\end{prop}

\begin{proof}
Take $F(n)=n.n!$. Now $G$ acts by conjugation on $S$, so $|G:C_G(S)|\le n!$.
Thus, $|C_G(S)|\ge n$. Now by hypothesis, no involution belongs to $C_G(S)$,
so $C_G(S)$ is a group of odd order. Thus the assumptions of
Theorem~\ref{t:mp2} are satisfied.
\end{proof}

\subsection{Embedding in groups whose power graph has a perfect matching}

\begin{thm}
Let $G$ be a finite group of even order, and suppose that the number of
elements of $G$ not matched in a matching of maximum size in $P(G)$ is $s$.
If $p$ is an odd prime greater than $s$, then $G\times C_p$ has a perfect
matching.
\end{thm}
\begin{proof}
Let $t_1,\ldots,t_s$ be the elements unmatched in some matching of maximum
size in $P(G)$. We know that without loss of generality we can assume that
$t_1,\ldots,t_s$ are involutions. (The set of unmatched vertices can be taken
to be a subset of $\{g\in G:g^2=1\}$, and the identity can be matched to any
other vertex.) Note that $s$ is even.

Take $p>s$, and let $x$ be a generator of $C_p$ in the group $G\times C_p$.
Let $A_0=\langle x\rangle\setminus\{1\}$, and for $1\le i\le s$ let
$A_i=A_0t_i$. Each set $A_i$ for $0\le i\le s$ induces a complete graph in
$P(G\times C_p)$, and we have all possible edges between $A_0$ and $A_i$
for $i>0$. Moreover, $t_i$ is joined to every vertex in $A_i$. Also,
$|A_i|=p-1$ for all $i$.

Choose an edge from $t_i$ to a vertex in $A_i$ for each $i$ and add to the
matching on $G$. There remain $p-2$ unmatched vertices in $A_i$; choose one,
and match it to a vertex in $A_0$, using distinct vertices for different $i$.
This leaves $p-3$ unmatched vertices in $A_i$, an even number, and
$p-1-s$ unmatched vertices in $A_0$, also an even number since $s$ is even.
So we can extend the matching by pairing up the unmatched vertices in $A_i$
for all $i$.

Finally, the vertices not yet matched come in inverse pairs, since they lie
outside the union of the subgroups $G$ and $\langle xt_i\rangle$; so we can
match each remaining vertex with its inverse.

\end{proof}

As a companion piece we have the following:

\begin{thm}
Let $G$ be a finite group of odd order. Then $P(G\times C_2)$ has a perfect
matching.
\end{thm}

\begin{proof} $G\times C_2$ has a unique involution.
\end{proof}

\subsection{$2$-groups}

The following theorem characterises the $2$-groups having perfect matchings in their power graphs.

\begin{thm}
Let $G$ be a finite group with $|G|=2^n$. Then $P(G)$ has a perfect matching if and only if $G$ is cyclic or generalized quaternion.
\end{thm}

\begin{proof}
We have $|O(G)|=|O(G_G(S))|=1$, so Theorems~\ref{t:mp1} and \ref{t:mp2} show
that $G$ has a perfect matching if and only if it has a unique involution. The
$2$-groups with unique involution are the cyclic and generalized quaternion
groups.
\end{proof}

\section{A number-theoretic result}

The functions $\tau(n)$ (the number of divisors of $n$) and $\phi(n)$ (Euler's
totient function) are two of the best-studied in number theory. The result we
require about them is elementary, but we have not found a proof in the
literature.

\begin{thm}
Let $n$ be a positive integer. If $n\ge30$, then $\tau(n)<\phi(n)$.
\label{t:nt}
\end{thm}

The proof depends on the formulae for these functions: if
$n=\displaystyle{\prod_{i=1}^rp_i^{a_i}}$, where $p_1,\ldots,p_r$ are distinct
primes and $a_1,\ldots,a_r$ are positive integers, then
\begin{enumerate}
\item $\tau(n)=\displaystyle{\prod_{i=1}^r(a_i+1)}$,
\item $\phi(n)=\displaystyle{\prod_{i=1}^rp_i^{a_i-1}(p_i-1)}$.
\end{enumerate}

We use the following technical lemma:

\begin{lem}
Let $p$ be a prime, and $a$ a positive integer.
\begin{enumerate}
\item If $(p,a)\notin\{(2,1),(2,2)\}$, then $p^{a-1}(p-1)\ge a+1$, with
equality only  if $(p,a)\in\{(2,3),(3,1)\}$.
\item If $p\ne2$ and $(p,a)\ne(3,1)$, then $p^{a-1}(p-1)\ge2(a+1)$, with
equality only if $(p,a)=(5,1)$.
\end{enumerate}
\end{lem}

\begin{proof}
The function $f(x)=p^{x-1}(p-1)-(x+1)$ has derivative
$f'(x)=p^{x-1}(p-1)\log p-1$, which is positive for $x\ge1$ if $p\ne2$,
and for $x\ge2$ if $p=2$. So for each $p$ we only have to check the smallest
values of $x$.
\end{proof}

\paragraph{Proof of the theorem}
To prove the theorem, we see that if $n$ is odd or divisible by $8$,
then $\phi(n)\ge\tau(n)$, with strict inequality if the factorization includes
$2^4$, $3^2$, or a prime larger than $3$. If $n$ is exactly divisible by $2^a$
with $a=1$ or $a=2$, then $2^{a-1}(2-1)\ge\frac{1}{2}(a+1)$, and so as long
as we have a factor $3^3$, $5^2$ or a prime greater than $5$ the strict
inequality holds. The cases $n=20$ and $n=36$ satisfy the conclusion. Thus,
the only cases for which it fails are $1, 2, 3, 4, 6, 8, 10, 12, 18, 24, 30$.

\medskip

The result we actually require is the following corollary of this theorem.
The \emph{independence number} $\alpha(\Gamma)$ of a graph $\Gamma$ is the
size of the largest set of vertices containing no edges.

\begin{cor}
Let $n$ be a positive integer. If $n\notin\{2,6\}$, then the independence
number of the power graph of the cyclic group $C_n$ is strictly less than
$\phi(n)$.
\end{cor}

\begin{proof}
In a cyclic group $C_n$, if two elements have the same order, then each is a
power of the other, so they are joined in the power graph. So an independent
set in the power graph has at most one element of each possible order, and
its cardinality is at most $\tau(n)$. By Theorem~\ref{t:nt}, the conclusion
holds if $n>30$; it is easily checked directly for smaller values of $n$.
\end{proof}

\begin{rk} In fact, it is easy to see that the independence number of
$P(C_n)$ is the size of the largest antichain in the lattice of divisors
of $n$. If $n$ is a product of $m$ primes (not necessarily distinct), then 
an antichain of maximum size is obtained by taking all distinct products of
$\lfloor m/2\rfloor$ primes, or all distinct products of $\lceil m/2\rceil$
primes. (This extension of the celebrated Sperner lemma was
proved by de Bruijn \emph{et al.}~\cite{bek}.) This fact can be used to
simplify the calculations in the Corollary.
\end{rk}

\section{The matching number of the enhanced power graph}

Recall that the \emph{enhanced power graph} $P_e(G)$ of a finite group $G$
is the graph with vertex set $G$ in which two vertices $x$ and $y$ are joined
if there exists $z$ such that both $x$ and $y$ are powers of $z$ (in other
words, if $\langle x,y\rangle$ is cyclic). So the enhanced power graph contains
the power graph as a spanning subgraph, and its matching number is at least
as great as that of the power graph.

From our earlier work, there are several cases where equality holds:
\begin{enumerate}
\item If $|G|$ is odd, then the power graph has a matching covering all
but one vertex; the same is true of the enhanced power graph.
\item If the power graph of $G$ has a perfect matching, then so does the
enhanced power graph.
\item Examining the proof of the formula for the matching number of the
power graph of a nilpotent group (Theorem~\ref{t:nilp}), we see that the
same formula holds for the enhanced power graph.
\end{enumerate}

In fact, we are going to prove that the matching numbers are always equal,
even in cases where we cannot compute them:

\begin{thm}
Let $G$ be a finite group. Then the matching numbers of the power graph and
the enhanced power graph of $G$ are equal.
\label{t:enhanced}
\end{thm}

\begin{proof}
Let $G$ be any finite group. Choose a matching $M$ of maximum size in the
enhanced power graph. If all its edges belong to the power graph, there is
nothing to prove. Otherwise, we are going to change $M$ to $M'$ so that $M'$
is a matching of the same size and has one fewer edge which doesn't belong
to the power graph.

So let $\{g,h\}$ be an edge of the matching $M$ which belongs to the enhanced
power graph but not to the power graph. Choose this edge so that
$\lcm(o(g),o(h))$ is as large as possible. Let $l$ be this lcm. Then
$\langle g,h\rangle=C$ is a cyclic group of order $l$. Let
$x_1,\ldots,x_{\phi(l)}$ be the generators of $C$. They are joined to all
vertices in $C$ in the power graph.

Assume first that least one of $x_1,\ldots,x_{\phi(l)}$, say $x_i$, is
not covered by the edges of $M$. Then we can replace the edge $\{g,h\}$
by the edge $\{g,x_i\}$, which is an edge of the power graph.

\medskip

So we can assume that all of $x_1,\ldots,x_{\phi(l)}$ are covered by edges
in $M$. Let $\{x_i,y_i\}$ be an edge of $M$ for $i=1,\ldots,\phi(l)$.

For each $i$, there are three cases:
\begin{enumerate}
\item $x_i$ is a power of $y_i$;
\item $y_i$ is a power of $x_i$;
\item neither of the above.
\end{enumerate}

In case (a), $g$ and $h$ are powers of $x_i$, and hence also powers of $y_i$.
So we can replace the edges $\{g,h\}$ and $\{x_i,y_i\}$ by $\{g,x_i\}$ and
$\{h,y_i\}$, both of which are edges of the power graph.

In case (c), $\{x_i,y_i\}$ is an edge of the enhanced power graph but not of
the power graph, and $\lcm(o(x_i),o(y_i))>l$, contradicting the choice of the
edge $\{g,h\}$.

So we must be in case (b) for all $i$. This means that all of $y_1,\ldots,
y_{\phi(l)}$ belong to $C$.

\medskip

Now suppose that $l\notin\{2,6\}$. Then the independence
number of the power graph of $C$ is strictly smaller than $\phi(l)$; so the
set $\{y_1,\ldots,y_{\phi(l)}\}$ is not an independent set in the power graph,
and so it contains at least one edge, say $\{y_i,y_j\}$. In this case, we
replace the three edges $\{g,h\}$, $\{x_i,y_i\}$, $\{x_j,y_j\}$ by 
$\{g,x_i\}$, $\{h,x_j\}$, $\{y_i,y_j\}$, all edges of the power graph.

\medskip

Finally, the case $l=2$ is clearly impossible. If $l=6$, let
$C=\langle z\rangle$ be the cyclic group of order~$6$. There are just two 
nonedges of the power graph, namely $\{z^3,z^2\}$ and $\{z^3,z^4\}$;
without loss of generality, $\{g,h\}=\{z^2,z^3\}$. We have $\{x_1,x_2\}=
\{z,z^5\}$. Hence necessarily $\{y_1,y_2\}=\{1,z^4\}$. But this is an edge of
the power graph, so the argument in the preceding paragraph applies.
\end{proof}

\section{Conclusion and open problems}

The most important problem we have been unable to solve is the following.

\begin{prob}
\begin{enumerate}
\item Find the matching number of $P(G)$ for any finite group $G$.
\item Find a necessary and sufficient condition for a group $G$ to have a
perfect matching.
\end{enumerate}
\end{prob}

The preceding section gives an interesting light on the relation between the
power graph and the enhanced power graph. We mention the following known
result. A finite group is an \emph{EPPO group} (for \textbf{E}lements of
\textbf{P}rime \textbf{P}ower \textbf{O}rder if every element has prime power
order. The \emph{Gruenberg--Kegel graph} (or \emph{prime graph}) of a group
$G$ has vertex set the set of prime divisors of $|G|$, with an edge from $p$
to $q$ if and only if $G$ contains an element of order $pq$.

\begin{thm}
The following conditions on a finite group $G$ are equivalent:
\begin{enumerate}
\item $P(G)=P_e(G)$;
\item $G$ is an EPPO group;
\item the Gruenberg--Kegel graph of $G$ is null.
\end{enumerate}
\end{thm}

The study of groups satisfying this condition was begun by Higman~\cite{higman}
in 1957, and all simple EPPO groups were found by Suzuki~\cite{suzuki} in 1962;
the complete determination of these groups was given by Bannuscher and Tiedt in
1994~\cite{bt}.

The following problem comprises a generalization of this theorem.

\begin{prob}
Let $p$ be a monotone graph parameter (that is, if $\Gamma$ is a spanning
subgraph of $\Delta$ then $p(\Gamma)\le p(\Delta)$). Determine the finite groups
for which $p(P(G))=p(P_e(G))$.
\end{prob}

Theorem~\ref{t:enhanced} shows that, if $p$ is the matching number, then the
solution is ``all finite groups''. Also, it is easy to show that, if $p$ is
the clique number, then the solution is ``all groups where the largest order
of an element is a prime power''.

\section*{Acknowledgments}
The author Swathi V V acknowledges the support of Council of Scientific and Industrial Research, India (CSIR) (Grant No-09/874(0029)/2018-EMR-I), and DST, Government of India,`FIST' (No.SR/FST /MS-I/2019/40).

The collaboration of the authors was made possible by the Research Discussion
on Graphs and Groups (RDGG) at CUSAT, Kochi, India, organised by Vijayakumar
Ambat and Aparna Lakshmanan S. We are grateful to them for this opportunity,

\end{document}